\documentclass[a4paper, 11pt]{article}
\usepackage{mathrsfs}
\usepackage{amsmath}

\usepackage{amsthm}
\usepackage{array}
\usepackage{cases}
\makeatletter
\def\th@plain{%
  \upshape %\itshape % body font
}
\makeatother

\makeatletter
\renewenvironment{proof}[1][\proofname]{\par
  \pushQED{\qed}%
  \normalfont \topsep6\p@\@plus6\p@\relax
  \trivlist
  \item[\hskip\labelsep
        \bfseries
    #1\@addpunct{.}]\ignorespaces
}{%
  \popQED\endtrivlist\@endpefalse
}
\makeatother

\newtheorem{thm}{Theorem}[section]
\newtheorem{cor}[thm]{Corollary}

\newtheorem{lem}[thm]{Lemma}

\numberwithin{equation}{section}

\newcommand{\C}{$\mathcal {C}$}
\newcommand{\R}{$\mathcal {R}$}

\usepackage[varg]{txfonts}
\usepackage{graphicx, epsfig, subfigure}
\usepackage{enumerate} %% 改变列表标号样式宏包 其后可接选项[a,A,i,I,1]
\usepackage[square, numbers, sort&compress]{natbib} %% 支持引用的宏包
\ifx\pdfoutput\undefined
 \usepackage[dvipdfm,%
  pdfstartview=FitH, %% FitH表示打开pdf文件的时候自动缩放页面大小到视窗大小，Fits表示页宽适应windows大小
  bookmarks=true,%
  bookmarksnumbered=true, %% 把章节的标号放入书签中，缺省为flase
  bookmarksopen=true, %% 打开书签树
  plainpages=false,%
  pdfpagelabels,%
  colorlinks=true, %% 彩色文字链接打开, 去掉方框
  linkcolor=blue, %% 目录链接的颜色为蓝色
  citecolor=blue,%
  urlcolor=black,
  pdfborder=001]{hyperref}
  \AtBeginDvi{}  %  GBK  ->  Unicode
\else
 \usepackage[pdftex,%
  pdfstartview=FitH, %% FitH表示打开pdf文件的时候自动缩放页面大小到视窗大小，Fits表示页宽适应windows大小
  bookmarks=true,%
  bookmarksnumbered=true, %% 把章节的标号放入书签中，缺省为flase
  bookmarksopen=true, %% 打开书签树
  plainpages=false,%
  pdfpagelabels,%
  colorlinks=true, %% 彩色文字链接打开
  linkcolor=blue, %% 目录链接的颜色为蓝色
  citecolor=blue,%
  urlcolor=black,
  pdfborder=001]{hyperref}%%% 不能和cite宏包同时使用
\fi

\numberwithin{equation}{section}

\begin{document}
\title{($1,\lambda$)-embedded graphs and the acyclic edge choosability\footnotetext{Emails: sdu.zhang@yahoo.com.cn (X. Zhang), gzliu@sdu.edu.cn (G. Liu), jlwu@sdu.edu.cn (J.-L. Wu).}\thanks{This research is partially supported by Graduate Independent Innovation Foundation of Shandong University (No. yzc10040) and National Natural Science Foundation of China (No. 10971121, 11026184, 61070230).}}
\author{Xin Zhang\thanks{The first author is under the support from The Chinese Ministry of Education Prize for Academic Doctoral Fellows}, Guizhen Liu, Jian-Liang Wu\\[.5em]
{\small School of Mathematics, Shandong University, Jinan 250100, P. R. China}\\
}

\date{}
\maketitle

\begin{abstract}
A ($1,\lambda$)-embedded graph is a graph that can be embedded on a surface with Euler characteristic $\lambda$ so that each edge is crossed by at most one other edge. A graph $G$ is called $\alpha$-linear if there exists an integral constant $\beta$ such that $e(G')\leq \alpha v(G')+\beta$ for each $G'\subseteq G$. In this paper, it is shown that every ($1,\lambda$)-embedded graph $G$ is 4-linear for all possible $\lambda$, and is acyclicly edge-$(3\Delta(G)+70)$-choosable for $\lambda=1,2$.\\[.5em]
\end{abstract}

\maketitle

\section{Introduction and basic definitions}

In this paper, all graphs considered are finite, simple and undirected. Let $G$ be a graph, we use $V(G)$, $E(G)$, $\delta(G)$ and $\Delta (G)$ to denote the vertex set, the edge set, the minimum degree and the maximum degree of a graph $G$. Let $e(G)=|E(G)|$ and $v(G)=|V(G)|$. Moreover, for embedded graph $G$ (i.e., a graph that can be embedded on a surface), by $F(G)$ we denote the face set of $G$. Let $f(G)=|F(G)|$. The girth $g(G)$
of a graph $G$ is the length of the shortest cycle of $G$. A $k$-, $k^+$- and $k^-$-vertex (or face) is a vertex (or face) of degree $k$, at least $k$ and at most $k$, respectively. A graph $G$ is called $\alpha$-linear if there exists an integral constant $\beta$ such that $e(G')\leq \alpha v(G')+\beta$ for each $G'\subseteq G$. Furthermore, if $\beta\geq 0$, then $G$ is said to be $\alpha$-nonnegative-linear; and if $\beta<0$, then $G$ is said to be $\alpha$-negative-linear. For other undefined concepts we refer the reader to \cite{Bondy}.

A mapping $c$ from $E(G)$ to the sets of colors $\{1,\cdots,k\}$ is called a proper edge-$k$-coloring of $G$ provided that any two adjacent edges receive different colors. A proper edge-$k$-coloring $c$ of $G$ is called an acyclic edge-$k$-coloring of $G$ if there are no bichromatic cycles in $G$ under the coloring $c$. The smallest number of colors such that $G$ has an acyclic edge coloring is called the \emph{acyclic edge chromatic number} of $G$, denoted by $\chi'_a(G)$. A graph is said to be acyclic edge-$f$-choosable, whenever we give a list $L_e$ of $f(e)$ colors to each edge $e\in E(G)$, there exists an acyclic edge-$k$-coloring of $G$, where each element is colored with a color from its own list. If $|L_e|=k$ for edge $e\in E(G)$, we say that $G$ is acyclicly edge-$k$-choosable. The minimum integer $k$ such that $G$ is acyclicly edge-$k$-choosable is called the \emph{acyclic edge choice number} of $G$, denoted by $\chi'_{c}(G)$.

Acyclic coloring problem introduced in \cite{Grunbam} has been extensively studied in many papers. One of the famous conjectures on the acyclic chromatic index is due to Alon, Sudakov and Zaks \cite{ASZ}. They conjectured that $\chi'_a(G)\leq \Delta(G)+2$ for any graph $G$. Alon et al.\cite{Alon} proved that $\chi'_a(G)\leq
64\Delta(G)$ for any graph $G$ by using probabilistic arguments. This bound for arbitrary graph was later improved to $16\Delta(G)$ by Molloy and Reed \cite{Molloy} and recently improved to $9.62\Delta(G)$ by Ndreca et al.\cite{Ndreca}.
In 2008, Fiedorowicz et al.\cite{Fiedorowicz} proved that $\chi'_a(G)\leq 2\Delta(G)+29$ for
each planar graph $G$ by applying a combinatorial method. Nowadays, acyclic coloring problem has attracted more and more attention since Coleman et al.\cite{Col.1, Col.2} identified acyclic coloring as the model for computing a Hessian matrix via a substitution method. Thus to consider the acyclic coloring problems on some other special classes of graphs seems interesting.

A graph is called $1$-planar if it can be drawn on the plane so that each edge is crossed by at most one other edge.
The notion of $1$-planar-graph was introduced by Ringel\cite{Ringel} while trying to simultaneously color the vertices and faces of a planar graph such that any pair of adjacent/incident elements receive different colors. In fact, from a planar graph $G$, we can construct a 1-planar graph $G'$
with its vertex set being $V(G)\cup F(G)$, and any two vertices of $G'$ being adjacent if and only if their corresponding elements in $G$ are adjacent or incident. Now we generalize this concept to ($1,\lambda$)-embedded graph, namely, a graph that can be embedded on a surface $S$ with Euler characteristic $\lambda$ so that each edge is crossed by at most one other edge. Actually, a ($1,2$)-embedded graph is a 1-planar graph. It is shown in many papers such as \cite{Fabrici} that $e(G)\leq 4v(G)-8$ for every 1-planar graph $G$. Whereafter, to determine whether the number of edges in the class of ($1,\lambda$)-embedded graphs is linear or not linear in the number of vertices for every $\lambda\leq 2$ might be interesting.

In this paper, we first investigate some structures of ($1,\lambda$)-embedded graph $G$ in Section 2 and then give a relationship among the three parameters $e(G)$, $v(G)$ and $g(G)$ of $G$, which implies that every ($1,\lambda$)-embedded graph is 4-linear for any $\lambda\leq 2$. In Section 3, we will introduce a linear upper bound for the acyclic edge choice number of the classes ($1,\lambda$)-embedded graphs with special given $\lambda$.

\section{The linearity of ($1,\lambda$)-embedded graphs}

Given a "good" graph $G$ (i.e., one for which all intersecting edges intersect in a single point and arise from four distinct vertices), the crossing number, denoted by $cr(G)$, is the minimum possible number of crossings with which the graph can be drawn.

%Now, we turn to consider the crossing number of $1$-embedded graphs. Here, we just determine an upon bound, which may be not tight, but very helpful
%for the next argument of this paper.

Let $G$ be a ($1,\lambda$)-embedded graph. In the following we always assume that $G$ has been embedded on a surface with Euler characteristic $\lambda$ so that each edge is crossed by at most one other edge and the number of crossings of $G$ in this embedding is minimum. Thus, $G$ has exactly $cr(G)$ crossings.
Sometimes we say such an embedding proper for convenience.

\begin{thm}\label{cr}
Let $G$ be a ($1,\lambda$)-embedded graph. Then $cr(G)\leq v(G)-\lambda$.
\end{thm}

\begin{proof}
Suppose $G$ has been properly embedded on a surface with Euler characteristic $\lambda$. Then for each pair of edges $ab,cd$ that cross each other at a crossing point $s$, their end vertices are pairwise distinct. For each such pair, we add new edges $ac,cb,bd,da$ (if it does not exist originally) to close $s$, then arbitrarily delete one edge $ab$ or $cd$ from $G$. Denote the resulting graph by $G^*$ and then we have $cr(G^*)=0$. By Euler's formula $v(G^*)-e(G^*)+f(G^*)=\lambda$ and the well-known relation $\sum_{v\in V(G^*)}d_{V(G^*)}(v)=\sum_{f\in F(G^*)}d_{V(G^*)}(f)=2e(G^*)$, $f(G^*)\leq 2v(G^*)-2\lambda$. Since each crossing point $s$ (note that $s$ is not a real vertex in $G$) lies on a common boundary of two faces of $G^*$ and each face of $G^*$ is incident with at most one crossing point (recall the definition of $G^*$), we deduce that $2cr(G)\leq f(G^*)$. Since $v(G)=v(G^*)$, we have $cr(G)\leq
\frac{f(G^*)}{2}\leq v(G^*)-\lambda=v(G)-\lambda$ in final.
\end{proof}

\begin{thm} \label{edge}
Let $G$ be a ($1,\lambda$)-embedded graph with girth at least $g$. Then $e(G)\leq \frac{2g-2}{g-2}(v(G)-\lambda)$.
\end{thm}

\begin{proof}
Suppose $G$ has been properly embedded on a surface with Euler characteristic $\lambda$. Now for each pair of edges $ab,cd$ that cross each other, we arbitrarily delete one from $G$. Let $G'$ be the resulting graph. One can easily see that $cr(G')=0$. By Euler's formula $v(G')-e(G')+f(G')=\lambda$ and the relations $v(G')=v(G)$, $e(G')=e(G)-cr(G)$, we have
\begin{equation} \label{1}
v(G)-e(G)+f(G')=v(G')-e(G')+f(G')-cr(G)=\lambda-cr(G)
\end{equation}
and
\begin{equation} \label{2}
\sum_{f\in F(G')}d_{G'}(f)=2e(G')=2(e(G)-cr(G))\geq g(G')f(G')\geq g(G)f(G')\geq g\cdot f(G').
\end{equation}
Now combine equations (\ref{1}) and (\ref{2}) together, we immediately have $e(G)\leq \frac{g}{g-2}(v(G)-\lambda)+cr(G)\leq
\frac{2g-2}{g-2}(v(G)-\lambda)$ by Theorem \ref{cr}.
\end{proof}

By Theorem \ref{edge}, the following two corollaries are natural.

\begin{cor}
Every ($1,\lambda$)-embedded graph is $4$-linear for any $\lambda\leq 2$.
\end{cor}

\begin{cor}
Every triangle-free ($1,\lambda$)-embedded graph is $3$-linear for any $\lambda\leq 2$.
\end{cor}

%Note that $\lim\limits_{g\rightarrow \infty}\frac{2g-2}{g-2}=2$. Hence there is no constant $\eta$ such that every ($1,\lambda$)-embedded graph with girth at least $\eta$ is $2$-linear for every $\lambda\leq 2$.

\section{Acyclic edge choosability of ($1,\lambda$)-embedded graphs}

In this section we mainly investigate the acyclic edge choosability of ($1,\lambda$)-embedded graphs with special given $\lambda$. In \cite{Fiedorowicz}, Fiedorowicz et al. proved the following two results.

\begin{thm} If $G$ is a graph such that $e(G')\leq
2v(G')-1$ for each $G'\subseteq G$, then $\chi'_a(G)\leq
\Delta(G)+6$.
\end{thm}

\begin{thm} If $G$ is a graph such that $e(G')\leq
3v(G')-1$ for each $G'\subseteq G$, then $\chi'_a(G)\leq
2\Delta(G)+29$.
\end{thm}

In fact, these two theorems respectively imply that the acyclic edge chromatic number of $2$-negative-linear graph $G$ is at most $\Delta(G)+6$ and that the acyclic edge chromatic number of $3$-negative-linear graph $G$ is at most $2\Delta(G)+29$.

Note that every triangle-free ($1,\lambda$)-embedded graph is $3$-negative-linear for any $1\leq \lambda\leq 2$ by Theorem \ref{edge}. Hence the following corollary is trivial.

\begin{cor}
Let $G$ be a triangle-free ($1,\lambda$)-embedded graph with $1\leq \lambda\leq 2$. Then $\chi'_a(G)\leq 2\Delta(G)+29$.
\end{cor}

%\begin{corollary}
%Let $G$ be triangle-free 1-planar graph. Then $\chi'_a(G)\leq 2\Delta(G)+29$.
%\end{corollary}

The following main theorem in this section is dedicated to giving a linear upper bound for the acyclic edge choice number of $4$-negative-linear graphs.

\begin{thm}\label{4-linear}
If $G$ is a graph such that $e(G')\leq 4v(G')-1$ for each $G'\subseteq G$, then $\chi'_{c}(G)\leq 3\Delta(G)+70$.
\end{thm}

As an immediately corollary of Theorems \ref{edge} and \ref{4-linear}, we have the following result.

\begin{cor}
Let $G$ be a ($1,\lambda$)-embedded graph with $1\leq \lambda\leq 2$. Then $\chi'_c(G)\leq 3\Delta(G)+70$.
\end{cor}

Before proving Theorem \ref{4-linear}, we first show an useful structural lemma.

\begin{lem}\label{configurations}
Let $G$ be a graph such that $e(G)\leq 4v(G)-1$ and $\delta(G)\geq 4$, Then at least one of the following
configurations occurs in $G$:\\
\noindent \rm ($\mathcal {C}$1) a $4$-vertex adjacent to a $19^-$-vertex;\\
\noindent \rm ($\mathcal {C}$2) a $5$-vertex adjacent to two $19^-$-vertices;\\
\noindent \rm ($\mathcal {C}$3) a $6$-vertex adjacent to four $19^-$-vertices;\\
\noindent \rm ($\mathcal {C}$4) a $7$-vertex adjacent to six $19^-$-vertices;\\
\noindent \rm ($\mathcal {C}$5) a vertex $v$ such that $20\leq d(v)\leq 22$ and at least $d(v)-3$ of its neighbors are $7^-$-vertices;\\
\noindent \rm ($\mathcal {C}$6) a vertex $v$ such that $23\leq d(v)\leq 25$ and at least $d(v)-2$ of its neighbors are $7^-$-vertices;\\
\noindent \rm ($\mathcal {C}$7) a vertex $v$ such that $26\leq d(v)\leq 28$ and at least $d(v)-1$ of its neighbors are $7^-$-vertices;\\
\noindent \rm ($\mathcal {C}$8) a vertex $v$ such that $29\leq d(v)\leq 31$ and all its neighbors are $7^-$-vertices;\\
\noindent \rm ($\mathcal {C}$9) a vertex $v$ such that at least $d(v)-7$ of its neighbors are $7^-$-vertices and at least one of them is of degree $4$.
\end{lem}

\begin{proof}
Suppose, to the contrary, that none of the nine configurations occurs in $G$. We assign to each vertex $v$ a charge $w(v)=d(v)-8$, then $\sum_{v\in V(G)}w(v)=\sum_{v\in V(G)}(d(v)-8)\leq -2$. In the following, we will reassign a new charge denoted by $w'(x)$ to each $x\in V(G)$ according to some discharging rules. Since our rules only move charges around, and do not affect the sum, we have
\begin{equation}\label{21}
\sum_{v\in V(G)}w'(v)=\sum_{v\in V(G)}w(v)\leq -2.
\end{equation}
We next show that $w'(v)\geq 0$ for each $v\in V(G)$, which leads to a desired contradiction. We say a vertex big (resp.\,small) if it is a $20^+$-vertex (resp.\,$7^-$-vertex). The discharging rules are defined as follows.

(\R1) Each big vertex gives $1$ to each adjacent $4$-vertex.\\
\indent(\R2) Each big vertex gives $\frac{3}{4}$ to each adjacent vertex of degree between 5 and 7.

Let $v$ be a $4$-vertex. Since (\C1) does not occur, $v$ is adjacent to four big vertices. So $v$ totally receives $4$ by
(\R1). This implies that $w'(v)=w(v)+4=d(v)-4=0$. Similarly, we can also prove the nonnegativity of $w'(v)$ if $v$ is a $k$-vertex where $5\leq k\leq 7$.
Let $v$ be a $k$-vertex where $8\leq k\leq 19$. Since $v$ is not involved in the discharging rules, $w'(v)=w(v)=d(v)-8\geq
0$.
Let $v$ be a $k$-vertex where $20\leq k\leq 22$. If $v$ is adjacent to a $4$-vertex, then $v$ is adjacent to at most $d(v)-8$ small vertices since (\C9) does not occur. Since $v$ sends each small vertex at most $1$ By (\R1) and (\R2), $w'(v)\geq w(v)-(d(v)-8)=0$. If $v$ is adjacent to no $4$-vertices, then $v$ sends each small vertex $\frac{3}{4}$ by (\R2). Since (\C5) does not occur either, $v$ is adjacent at most $d(v)-4$ small vertices. So $w'(v)\geq
w(v)-\frac{3}{4}(d(v)-4)=\frac{1}{4}(d(v)-20)\geq 0$. By similar arguments as above, we can also respectively show the nonnegativity of $w'(v)$
if $v$ is a $k$-vertex where $k\geq 23$.
\end{proof}

\noindent \emph{\textbf{Proof of Theorem} $\ref{4-linear}$}. Let $K$ stands for $3\Delta(G)+70$. We prove the theorem by contradiction. Let $G$ be a counterexample to the theorem with the number of edges as small as possible. So there exists a list assignment $L$ of $K$ colors such that $G$ is not acyclicly edge-$L$-choosable. For each coloring $c$ of $G$, we define $c(uv)$ to be the color of edge $uv$ and set $C(u)=\{c(uv)|uv\in E(G)\}$ for each vertex $u$. For $W\subseteq V(G)$, set $C(W)=\bigcup_{w\in W}C(w)$. If $uv\in E(G)$, we let $W_G(v,u)$ stands for the set of neighbors $w$ of $v$ in $G$ such that $c(vw)\in C(u)$. Now, we first prove that $\delta(G)\geq 4$.

Suppose that there is a $3$-vertex $v\in V(G)$. Denote the three neighbors of $v$ by $x$, $y$ and $z$. Then by the minimality of $G$, the graph $H=G-ux$ is acyclicly edge-$L$-choosable. Let $c$ be an acyclic edge coloring of $H$. We can extend $c$ to $uv$ by defining a list of available colors for $uv$ as follows: $$A(uv)=L(uv)\backslash (C(x)\cup C(y)\cup C(z)).$$ Since $|C(x)|\leq \Delta(G)-1$, $|C(y)|\leq \Delta(G)$ and $|C(z)|\leq \Delta(G)$, we have $|A(uv)|\geq K-3\Delta(G)+1>0$. So we can color $uv$ by a color in $A(uv)\subseteq L(uv)$, a contradiction. Similarly, one can also prove the absences of 1-vertices and 2-vertices in $G$. Hence $\delta(G)\geq 4$. Then by Lemma \ref{configurations}, $G$
contains at least one of the configurations (\C1)-(\C9). In the following, we only show that if one of the configurations (\C4), (\C5) and (\C9) appears, then we would get a contradiction. That is because the proofs are similar and easier for another six cases.

\noindent \textbf{Configuration} (\C4): Suppose that there is a $7$-vertex $v$ who is adjacent to six $19^-$-vertices, say $x_1,x_2,\cdots,x_6$. Denote another one neighbor of $v$ by $x_7$. Then by the minimality of $G$, the graph $H=G-vx_7$ is acyclicly edge-$L$-choosable. Let $c$ be an acyclic edge coloring of $H$. Suppose $c(vx_j)\not\in C(x_7)$ for some $1\leq j\leq 6$. Then we can extend $c$ to $vx_7$ by defining a list of available colors for $vx_7$ as follows: $$A(vx_7)=L(vx_7)\backslash \bigcup_{1\leq i\neq j\leq 7} C(x_i).$$
Since $|C(x_i)|\leq \min\{19,\Delta(G)\}$ for every $1\leq i\leq 6$ and $|C(x_7)|\leq \Delta(G)-1$, we have $|A(vx_7)|\geq K-\min\{\Delta(G)+94,6\Delta(G)-1\}=\max\{2\Delta(G)-24,-3\Delta(G)+69\}>0$. So we can color $vx_7$ by a color in
$A(vx_7)\subseteq L(vx_7)$, a contradiction. Thus we shall assume that $c(vx_j)\in C(x_7)$ for every $1\leq j\leq 6$. This implies that $|\bigcup_{1\leq i\leq 7} C(x_i)|\leq \Delta(G)-1+6\times \min\{19,\Delta(G)\}-6=\min\{\Delta(G)+107,7\Delta(G)-7\}$. Now we extend $c$ to $vx_7$ by defining a list of available colors for $vx_7$ as follows: $$A(vx_7)=L(vx_7)\backslash \bigcup_{1\leq i\leq 7} C(x_i).$$ Note that $|A(vx_7)|\geq K-\min\{\Delta(G)+107,7\Delta(G)-7\}=\max\{2\Delta(G)-37,-4\Delta(G)+77\}>0$. So we can again color $vx_7$ by a color in
$A(vx_7)\subseteq L(vx_7)$, also a contradiction.

\noindent \textbf{Configuration} (\C5): If there is a vertex $v$ such that $20\leq d(v)\leq 22$ and at least $d(v)-3$ of its neighbors are $7^-$-vertices. Without loss of generality, we assume that $d(v)=22$ and that $v$ have nineteen $7^-$-neighbors. Denote another three neighbors of $v$ by $x,y$ and $z$. Choose one $7^-$-neighbor, say $u$, of $v$. Without loss of generality, we assume that $d(u)=7$. Then by the minimality of $G$, the graph $H=G-uv$ is acyclicly edge-$L$-choosable. Let $c$ be an acyclic edge coloring of $H$. Then we can extend $c$ to $uv$ by defining a list of available colors for $uv$ as follows: $$A(uv)=L(uv)\backslash \{C(u)\cup C(v)\cup C(x)\cup C(y)\cup C(z)\cup C(W_H(v,u)\}.$$ Since $c$ is an acyclic (and thus it is proper), $|W_H(v,u)|\leq d(u)-1=6$. Since $|C(x)|\leq \Delta(G)$, $|C(y)|\leq \Delta(G)$, $|C(z)|\leq \Delta(G)$ and $|C(w)|\leq 7$ for each $w\in W_H(v,u)$, we have $|A(uv)|\leq K-(3\Delta(G)+6+6\times 6+21-9)>0$. So we can color $uv$ by a color
in $A(uv)\subseteq L(uv)$, a contradiction.

\noindent \textbf{Configuration} (\C9) If there is a vertex $v$ such that at least $d(v)-7$ of its neighbors are $7^-$-vertices and at least one of them is of degree $4$, say $u$. Denote another three neighbors of $u$ by $x,y$ and $z$. Let $C_1=\{c(vw)|w\in N_H(v)\ and\ d_H(w)>7\}$ and $C_2=\{c(vw)|w\in N_H(v)\ and\ d_H(w)\leq7\}$. Then $|C_1|\leq 7$. By the minimality of $G$, the graph $H=G-uv$ is acyclicly edge-$L$-choosable. Let $c$ be an acyclic edge coloring of $H$. Suppose $C(u)\cap C_1\not=\emptyset$. Without loss of generality, we assume that $c(ux)\in C_1$. Now we erase the color of the edge $ux$ from $c$ and recolor it from the list defined as follows: $$A(ux)=L(ux)\backslash \{C(x)\cup C(y)\cup C(z)\cup C_1\}.$$ Since $|C(x)|\leq \Delta(G)$, $|C(y)|\leq \Delta(G)$, $|C(z)|\leq \Delta(G)$ and $|C_1|\leq 7$, we have $|A(ux)|\geq K-(3\Delta(G)+7)>0$. Note that $A(ux)$ is just a sub-list of the original list given at the beginning of the proof and the new color of $ux$ preserves the acyclicity of the coloring of $H$. So we can assume that $C(u)\cap C_1=\emptyset$. In this case, we can extend $c$ to the edge $uv$ by defining a list of available colors for $uv$ as follows: $$A(uv)=L(uv)\backslash \{C(u)\cup C_1\cup C_2\cup C(W_H(v,u))\}.$$ Since $C(u)\cap C_1=\emptyset$, we have $c(vw)\in C_2$ for each $w\in W_H(v,u)$ and thus $|C(W_H(v,u))|\leq 7d_H(u)=21$. Since $|C_1\cup C_2|=d_H(v)\leq \Delta(G)-1$ and $|C(u)|=3$, we have $|A(uv)|\geq K-(\Delta(G)+23)>0$. So we can color $uv$ by a color in
$A(uv)\subseteq L(uv)$. This contradiction completes the proof of Theorem \ref{4-linear}. $\hfill\square$

\end{document}